\documentclass[11pt]{amsart}
\usepackage[utf8]{inputenc}
\usepackage{amsmath,amssymb,amsthm,verbatim} 
\usepackage{longtable}
\usepackage{bbm}
\usepackage{graphicx}
\usepackage[dvipsnames]{xcolor}
\usepackage[pdftex, colorlinks, linkcolor=blue, citecolor=blue, urlcolor=blue]{hyperref}
\hypersetup{linkcolor=MidnightBlue, citecolor=MidnightBlue, urlcolor=MidnightBlue}
\usepackage{enumerate,subcaption,tikz} 
\usepackage[nameinlink]{cleveref}
\usepackage{geometry}
\usepackage{algorithm,algpseudocode}

\newtheorem{theorem}{Theorem}
\newtheorem{proposition}[theorem]{Proposition}
\newtheorem{lemma}[theorem]{Lemma}
\newtheorem{corollary}[theorem]{Corollary}

\theoremstyle{definition}
\newtheorem{definition}[theorem]{Definition}
\newtheorem{example}[theorem]{Example}
\crefname{ineq}{inequality}{inequalities}
\crefname{prop}{Proposition}{Propositions}
\creflabelformat{ineq}{#2\textup{(#1)}#3}

\newcommand{\RR}{\mathbbm{R}}
\newcommand{\ZZ}{\mathbbm{Z}}
\newcommand\N{{\mathbb N}}

\newcommand{\eins}[1][\mbox{}]{\ensuremath\mathbbm{1}_{#1}}
\newcommand{\bundle}{a}
\newcommand{\bundleStar}{\bundle^*}
\newcommand{\piInverse}{\pi^{-1}}
\newcommand{\decomposable}{A}
\DeclareMathOperator{\vertices}{vert}
\DeclareMathOperator{\valuationOperator}{val}
\newcommand{\valuation}[1][b]{\valuationOperator^{#1}}
\newcommand{\valuationtilde}[1][b]{\widetilde{\valuationOperator}^{#1}}
\DeclareMathOperator{\Valuation}{Val}
\newcommand{\weightOperator}{w}
\newcommand{\weight}[1][b]{\weightOperator^{#1}}
\newcommand{\weighttilde}[1][b]{\tilde{\weightOperator}^{#1}}
\newcommand{\product}[2]{\langle #1, #2 \rangle}
\newcommand{\demandOperator}{D}
\newcommand{\demand}[2]{\demandOperator({#1},#2)}
\newcommand{\price}{p}
\newcommand{\subd}{\Delta}

\newcommand{\partition}{S}

\newcommand{\rr}{r}
\newcommand{\point}{q}
\newcommand{\letter}{b}

\newcommand{\allocationOperator}{\mu}
\newcommand{\allocation}[1][b]{\allocationOperator^{#1}}
\newcommand{\setOfItems}{S}

\newcommand{\mset}[1]{\ensuremath{ \left\lbrace #1 \right\rbrace }}
\newcommand{\given}{\,\middle\vert\,}

\renewcommand{\mapsto}{\longmapsto}
\newcommand{\inner}[1]{\left\langle \, #1  \,\right\rangle}
\newcommand{\ma}{\begin{pmatrix} } 
	\newcommand{\trix}{\end{pmatrix}}
\newcommand{\sma}{\left(\begin{smallmatrix} } \newcommand{\strix}{\end{smallmatrix}\right)}

\newcommand\conv{\operatorname{conv}}

\DeclareMathOperator{\lift}{lift}
\DeclareMathOperator*{\argmax}{arg\,max}

\begin{document}
\title[Competitive equilibrium with quadratic pricing]{Competitive
  equilibrium always exists for combinatorial auctions with graphical pricing schemes}
\author[Brandenburg]{Marie-Charlotte Brandenburg}
\author[Haase]{Christian Haase}
\author[Tran]{Ngoc Mai Tran}

\date{}

\begin{abstract}
We show that a competitive equilibrium always exists in combinatorial auctions with anonymous graphical valuations and pricing, using discrete geometry. This is an intuitive and easy-to-construct class of valuations that can model both complementarity and substitutes, and to our knowledge, it is the first class besides gross substitutes that have guaranteed competitive equilibrium. We prove through counter-examples that our result is tight, and we give explicit algorithms for constructing competitive pricing vectors. We also give extensions to multi-unit combinatorial auctions (also known as product-mix auctions). Combined with theorems on graphical valuations and pricing equilibrium of Candogan, Ozdagar and Parrilo, our results indicate that quadratic pricing is a highly practical method to run combinatorial auctions.
\end{abstract}

\maketitle

\section{Introduction}
In a \emph{multi-unit combinatorial auction}, also known as product-mix auction \cite{Tran_Yu_Product-Mix_Auctions,Baldwin_Klemperer_Understanding_Preferences},
multiple agents can make simultaneous bids on multiple subsets of indivisible goods of distinct types. The running example in this paper is the Cutlery Auction of three items, a fork, a knife and a spoon, among three agents, Fruit, Spaghetti and Steak. Each is willing to pay at most 1 dollar for their favorite combination and no other: for Fruit, it is (knife, spoon), for Spaghetti, it is (fork, spoon), and for Steak, it is (fork, knife).

In the auctioneer's disposal is the ability to \emph{design} the auction. While this is a vast topic \cite{nisan2015algorithmic,borgers2015introduction}, for this paper we follow the setup of algorithmic auction design \cite{roughgarden2015prices,candogan_parrilo_pricing_equilibria}, namely, that the auctioneer can restrict the class of bid functions (valuations) that the agents can submit as well as the class of price functions that can be announced. The problem faced by the auctioneer is the following:
Given a bundle of items and an equilibrium concept $E$, find a valuation class $\mathcal{V}$ and a pricing class $\mathcal{P}$ such that for any auction with valuations in $\mathcal{V}$, $E$ is guaranteed to exist for some allocation $\allocationOperator$ and some pricing function $\price \in \mathcal{P}$. 

An ideal market would allocate the resources efficiently (maximize welfare) and set prices to achieve a competitive equilibrium where no participant wants to deviate. Existence of such equilibria is a function of the class of valuations $\mathcal{V}$, the pricing rules $\mathcal{P}$, and the definition of equilibrium. 

In this paper, we discuss two commonly studied notions equilibria: competitive equilibrium (CE) and pricing equilibrium (PE). The formal definitions are given in \Cref{sec:economics}. The difference lies in the involvement of the seller. For CE, the seller simply wants all other participants to be happy with their allocations at the announced price. Such allocations are called competitive allocations. For PE, the seller plays the same role as all the other agents, namely, they personally want to maximize their profit without caring for others. In this case, PE exists if the allocation and prices satisfy both the sellers' and all the other agents' agenda.

The simplest pricing scheme is linear pricing, where each type of item is assigned a price $p_i$, and the price of a bundle $S$ is $\sum_{i\in S}p_i$. For linear pricing,  both notions coincide and are known as \emph{the} Walrasian equilibrium. This is a classical concept in economics, see \cite{bichler2021walrasian} for a recent survey. In many combinatorial auctions, however, a Walrasian equilibrium may not exist. The Unimodularity Theorem \cite{danilov2004discrete} states that Walrasian equilibrium is guaranteed to exist at \emph{any} feasible bundle for \emph{any} auction with valuations in $\mathcal{V}$ if and only if $\mathcal{V}$ is the set of gross substitutes (GS) valuations.  

Unfortunately, GS valuations have several well-documented undesirable properties that hinder their applications \cite{roughgarden2015prices,candogan_parrilo_pricing_equilibria,leme2017gross}. First, the GS condition
imposes exponentially many linear inequalities on the valuation
$v$. It takes $O(2^{n} - poly(n))$ operations to check if an arbitrary
function $v$ is gross substitutes. Existing methods to
generate GS valuations from `unit' functions and elementary
operations with natural economics interpretations can only generate strict subclasses such as OXS and EAV. Second, the GS valuations cannot exhibit complementarity,
where items $i$ and $j$ can enhance each other, so the value of the
pair $v(ij)$ can be more than the sum of its parts $v(i) + v(j)$. This is the case for the Cutlery Auction, where all items
are complements. 

By the Unimodularity Theorem, deviating from gross substitutes means
deviating from guaranteed equilibrium \emph{under linear
	pricing}. Therefore, a major challenge in combinatorial auction is
to identify analogues of the Unimodularity Theorem for a more
practical class of valuations under a more general pricing rule
$\mathcal{P}$. Our main theorem establishes such a result for the
class of \emph{graphical valuations} with \emph{anonymous graphical
	pricing} under the concept of \emph{competitive equilibrium}.

\begin{definition}\label{defn:graphical}
	A valuation is an assignment 
	$\valuationOperator: 2^{[n]} \to \RR$. It is called \emph{graphical} if there exists a simple, undirected graph $G$ with vertices $V(G)=[n]$ and edges $E(G)$, and a vector $\weightOperator \in \RR^{ V(G) \sqcup E(G) }$ such that
	\[
	\valuationOperator(\setOfItems) = \sum_{i \in S} \weightOperator_i + \sum_{\substack{ij \in E(G) \\ i,j\in S}} \weightOperator_{ij}.
	\]
	for every $S\subseteq [n]$.
	The graph $G$ is called the \emph{underlying value graph} of the valuation. Similarly, a pricing function is called graphical if there exists a pricing vector $p$ such that the price of a bundle $S$ is given by 
	$
	\sum_{i \in S} \price_i + \sum_{\substack{ij \in E(G) \\ i,j\in S}} \price_{ij}.
	$
\end{definition}

We will always consider \emph{anonymous} pricing, i.e. each agent will be offered the same price for each bundle. We now present one of the main results of our paper:

\begin{theorem}\label{thm:main}
	If $\mathcal{V}$ is the set of graphical valuations with the complete graph $K_n$ as the underlying graph, then a
	competitive equilibrium with anonymous graphical pricing is guaranteed
	to exist at \emph{any} bundle $a^\ast \in \{0,1\}^n$ for
	\emph{any} auction with valuations in $\mathcal{V}$. 
\end{theorem}

Our theorem is tight, in that one cannot replace CE in the statement with PE, or replace graphical pricing with linear pricing. In particular, the Cutlery Auction is an instance of an auction with complete graphical valuations that does \emph{not} have a Walrasian equilibrium nor a pricing equilibrium with graphical pricing, but it does have a CE  with graphical pricing (cf. \Cref{example:triangle}). In other words, for complete graphical valuations, graphical pricing is \emph{truly necessary} to achieve competitive equilibrium, but not necessarily pricing equilibrium. 

We present two generalizations of this statement in \Cref{sec:results}. In \Cref{thm:main-general}, we allow the agents' graphs to be complete subgraphs $K_S$ where $S \subsetneq [n]$. \Cref{th:CE}  allows the auctioneer to sell more than one item per type.

Our proof is constructive, and not only so, we give two \emph{different} constructions of pricing functions that support the competitive equilibrium. This yields a practical algorithm for computing a quadratic CE price (cf. \Cref{alg:ce}). \Cref{thm:main} and the algorithm thus makes a strong case for implementing graphical valuations for combinatorial auctions.

We stress that \Cref{thm:main} is not obtained directly from the Unimodularity Theorem. 
Different proofs of the Unimodularity Theorem  \cite{danilov2004discrete,Baldwin_Klemperer_Understanding_Preferences,Tran_Yu_Product-Mix_Auctions} all trace back to the total unimodularity of a certain integer program. In contrast, our proof relies on the geometry of the \emph{correlation polytope}. \\ 
\\

\textbf{Paper organization.} Our paper is organized as follows. In \Cref{sec:economics,subsec:conv-geom} we collect necessary definitions from multi-unit combinatorial auctions and discrete convex geometry. We formulate the problems in \Cref{sec:ce.and.geometry} using the language of regular subdivisions and lattice polytopes. The main results are presented and discussed in \Cref{sec:results}, and proved in \Cref{sec:proofs}. \\
\\

\textbf{Notation.} 
Throughout this paper, except for examples, we fix an integer $n$ and denote $[n] = \mset{1,\dots,n}$. Let $G$ be a graph of $[n]$ vertices. For a subgraph $H \subseteq G$, let $V(H) \subseteq [n]$ denote its vertex set and $E(H) \subseteq \binom{[n]}{2}$ its edge set.  Let $d := n + |E(G)|$. The characteristic vector of $H$ is $\chi_H \in \{0,1\}^d$, where the first $n$ coordinates $((\chi_H)_i, i \in [n])$ are indicators for the vertices of $H$, while the next $\vert E(G)\vert $ coordinates $((\chi_H)_{ij}, ij \in E(G))$ are indicators for the edges of $H$. In general, a vector $\bundle\in\RR^d$ inherits the same indexing system, that is, the first $n$ coordinates are indexed by $ i \in [n]$, while the next $\vert E(G)\vert $ coordinates are indexed by $ij \in E(G)$. The projection of $\bundle \in \RR^d$ onto its first $n$ coordinates is denoted $\pi(a) = \bundleStar \in \RR^n$. 
\\

\section{Discrete convex geometry and economic equilibria}\label{sec:background}

The product-mix auction with linear pricing forms a fundamental connection between discrete convex geometry and economic equilibria. We recommend \cite{Baldwin_Klemperer_Understanding_Preferences} as an introduction for economists and \cite{Tran_Yu_Product-Mix_Auctions} for geometers. 
In this section, we extend this bridge to the case of product-mix auctions with nonlinear pricings.

\subsection{The economics setup}\label{sec:economics}

Consider an auction with $n$ types of
indivisible goods on sale to $m$ agents, where $\bundleStar_i \in \N$ is
the number of goods of type $i\in[n]$. 
Throughout this paper, we assume that each agent wants to buy at most one item of each type, even though the auctioneer might sell more than one item of a this type to the group of agents.

A graphical pricing function $\price: 2^{[n]}
\to \RR$ specifies the price to be paid for each bundle. Each
agent $b \in [m]$ has a graphical valuation function $\valuation: 2^{[n]} \to \RR$, where $\valuation(\setOfItems) \in \RR$ measures how much
the agent values the bundle $\setOfItems$. 

Following \cite{candogan_parrilo_pricing_equilibria}, we model the relations between the goods via a graph $G$ on $n$ vertices. Each vertex represents a type of goods, while edges between vertices model the existence of a relationship between types, such as complementarity and substitutability.
Given a subset of types $S\subseteq [n]$, we construct a vector $\bundle_S\in \{0,1\}^{[n] \sqcup E(G)}$, where the first $n$ coordinates are indexed by the vertices of $G$ and the following coordinates are indexed by the edges of $G$. Explicitly, we define $\bundle_S$ 
by
\[
a_{S,i} = \begin{cases}
	1 & \text{ if } i \in S \\
	0 & \text{otherwise}
\end{cases}, \qquad
a_{S,ij} = \begin{cases}
	1 & \text{ if } i,j \in S \text{ and } ij \in E(G) \\
	0 & \text{otherwise}
\end{cases}.
\]
We can then rewrite the valuation from \Cref{defn:graphical} as $\valuationOperator: \RR^{[n]+E(G)} \to \RR,$ $\valuationOperator(\bundle)~=~\inner{\weightOperator,\bundle}
$
and a price function as $ \inner{\price,\bundle}$.

Such a valuation function can be interpreted as follows:
If $\weightOperator_{ij} > 0$ for some edge $ij\in E(G)$, then agent $b$ views
items of types $i$ and $j$ as \emph{complementary}, i.e. is rather
interested in buying both items together than only a single one of
them. The higher the value of $\weightOperator_{ij}$, the higher is agent
$b$'s preference of buying both $i$ and $j$ at the same time.
Similarly, if $\weightOperator_{ij} < 0$, then items of types $i$ and $j$ are
viewed as \emph{substitutable}, i.e. agent $b$ prefers to buy only one
of the two items. 
Let $\pi = \pi_G$ denote the coordinate projection
\[
\pi: \ZZ^{[n]\sqcup E(G)} \to \ZZ^{[n]}
\]
that forgets the coordinates which correspond to edges in $G$.  \\

The auction works as follows. First, the agents submit their
valuations to the auctioneer, then the auctioneer announces the price $\price$ and an
allocation $\bundle^1,\dots,\bundle^m \in \{0,1\}^d$, where agent $b$ is assigned the bundle $\bundle^b$ and $\bundle = \sum_{b=1}^m \bundle^b \in\piInverse(\bundleStar)$. If all of the agents and the seller agree, then
agent $b$ gets bundle $\bundle^b$ and pays $\inner{\price,\bundle^b}$ to the seller. If one
of them does not agree, then in theory, `equilibrium failed' (and
perhaps the auctioneer should be fired). In practice, of course, many
things can happen. The bare-minimum goal for the auctioneer is to compute an allocation-pricing pair such that all of the market participants can agree to, that is, a non-trivial economic equilibrium is reached. 
It is assumed that agent $b$ is satisfied with the allocation-price
pair $((\bundle^1,\dots,\bundle^m),\price)$ when they are allocated a bundle that maximizes
their utility at this price. At a given price $\price$, the set of such
bundles for agent $b$ is their \emph{demand set}
\[ \demand{\valuation}{\price}  = \argmax_{S \subseteq [n]} \mset{
	\valuation(\bundle_S) - \inner{ \price,\bundle_S} }, \]
where $a_S \in \{0,1\}^{[n]\sqcup E(G)}$ is the vector corresponding to the set $S$.
Thus, agent $b$ is satisfied with the allocation-price pair when $\bundle^b \in
\demand{\valuation}{\price}$. When all agents are satisfied, we have a
\emph{competitive equilibrium}.

\begin{definition}[Competitive equilibrium]\label{def:ce}
	Let $(\valuation: b \in [m])$ be graphical valuations with a fixed value graph $G$. 
	We say that the corresponding auction has a \emph{competitive
		equilibrium at the allocation-price pair} $\left((\bundle^1,\dots,\bundle^m),\price \right)$ if  $\bundle^b \in \demand{\valuation}{\price}$ for all $b \in [m]$. The auction has a \emph{competitive
		equilibrium} at $a \in \ZZ^{[n]\sqcup E(G) }$ and a price $p\in\RR^{[n]+E(G)}$ if there exists an allocation $(a^1,\dots,a^m)$  such that  $\bundle = \sum \bundle^b $ and $\bundle^b \in \demand{\valuation}{\price}$ for all $b \in [m]$. \\
	We say that the auction has a \emph{competitive equilibrium (CE) at a bundle $\bundleStar \in \N^n$ of goods} if there exists some  $ \bundle \in\piInverse(\bundleStar)$ and some price $p\in\RR^{[n]+E(G)}$ at which the auction has a competitive equilibrium.
\end{definition}

In this setup, we only consider allocations that are \emph{complete}, that is, the seller must sell all items in the bundle. Note that the seller's revenue only depends on the sum of the bundles in the allocation, as $ \sum_{b=1}^m \inner{\price, \bundle^b } = \inner{ \price, \sum_{b=1}^m \bundle^b } = \inner{\price, \bundle } $.
Informally, a competitive equilibrium says that the agents are always satisfied.
A stronger concept is pricing equilibrium, where the seller is another market participant who also needs to be satisfied. A revenue-maximization seller wants to maximize revenue at a given price $\price$, that is,
\[ \demand{\text{seller}}{\price}  = \argmax_{\bundle \in \piInverse(\bundleStar)} \mset{  \inner{\price,\bundle}}. \]	
This leads to the definition of a pricing equilibrium {\cite[Def. 2.3]{candogan_parrilo_pricing_equilibria}}. 

\begin{definition}[Pricing equilibrium]
	Let $(\valuation: b \in [m])$ be graphical valuations. 
	We say that the corresponding auction has a \emph{pricing
		equilibrium (PE)} at the allocation-price pair $((\bundle^1,\dots,\bundle^m),\price)$ 
	if 
	$\bundle^b \in \demand{\valuation}{\price}$ for all $b \in [m]$  and in addition $\bundle = \sum_{b=1}^m \bundle^b \in \demand{\text{seller}}{\price}$.
\end{definition}

In between these two notions lies the \emph{optimal competitive equilibrium}: an allocation-price pair such that the seller's revenue is maximized among all $\bundle \in \piInverse(\bundleStar)$ at which a competitive equilibrium exists.
The Cutlery Auction is a simple auction that illustrates the difference between these concepts.
 
\begin{example}[An auction with CE but no PE for graphical pricing]\label{example:triangle}
	We consider the Cutlery Auction due to \cite[Example 3.2]{candogan_parrilo_trees}. 
	Let $G=K_3$ be the complete graph on three vertices $V(K_3)=\{A,B,C\}$, $m=3$ and $\bundleStar = (1,1,1)$. The agent's valuations are given by the weight vectors
	\[ \weight[1] = (0,0,0,1,0,0)^T, \quad \weight[2] = (0,0,0,0,1,0)^T, \quad \weight[3] = (0,0,0,0,0,1)^T, \]
	i.e. the first agent has weight one for edge $AB$, the second agent has weight one for edge $AC$, the third agent has weight one for edge $BC$ and all remaining weights are zero.
	If we pick the graphical price vector
	$\price = (0,0,0,1,1,1)^T,$
	then for each $j\in\{1,2,3\}$ we have
	\[
	\demand{\valuation}{\price}\supseteq\left\{
	\left(\begin{smallmatrix}  0\\0\\0\\0\\0\\0 \end{smallmatrix}\right),
	\left(\begin{smallmatrix}  1\\0\\0\\0\\0\\0 \end{smallmatrix}\right),
	\left(\begin{smallmatrix}  0\\1\\0\\0\\0\\0 \end{smallmatrix}\right),
	\left(\begin{smallmatrix}  0\\0\\1\\0\\0\\0 \end{smallmatrix}\right)
	\right\}.
	\]
	Thus, we can decompose $\bundle = (1,1,1,0,0,0) \in\piInverse(\bundleStar)$  by assigning one item to each agent, e.g.
	\[ \bundle^1 = (1,0,0,0,0,0)^T, \quad \bundle^2= (0,1,0,0,0,0)^T, \quad \bundle^3 = (0,0,1,0,0,0)^T, \]
	(where each agent is charged the price $0$) in order to achieve a competitive equilibrium. However, the seller's revenue at this price is $\product{\price}{\bundle^1+ \bundle^2 + \bundle^3}=0$. 
	
	If we instead decide to sell items $A$ and $B$ to agent $1$ and item $C$ to agent $2$, then this constitutes a competitive equilibrium in which the sellers revenue is $\product{\price}{(1,1,0,1,0,0)^T + (0,0,1,0,0,0)^T}=1$. In fact this is the maximum revenue that the seller can achieve under all allocations that constitute a competitive equilibrium, and thus induces an optimal CE. However, this does not constitute a pricing equilibrium, since assigning all items to one agent would increase the seller's profit to $3$ (but none of the agents would be happy with this outcome).
	It was shown in \cite[Example 3.13]{candogan_parrilo_pricing_equilibria}, that neither a pricing equilibrium nor a Walrasian equilibrium for this example exists. 
	We will continue with this in \Cref{ex:triangle-cont}. 
\end{example}

\subsection{Convex geometry setup}\label{subsec:conv-geom}
We now describe a strong connection between CE and discrete convex
geometry for graphical valuations and pricing. 

\begin{definition}
	The \emph{polytope of the graph $G$} is
	\[ P(G) = \conv\left(\mset{\bundle_{S}\in\RR^{[n] \sqcup E(G) } \given S \subseteq [n]}\right), \]
	where $\bundle_S$ is the characteristic vector of the set $S$ as described in \Cref{sec:economics}. \\
	Let $d~=~\vert[n]~\sqcup~E(G)\vert$ denote the dimension of $P(G)$. 
	By construction, $P(G)$ is a 0/1--polytope, and therefore
	\[ P(G)\cap\ZZ^d = \vertices(P(G)) = \mset{\bundle_{S}\in\RR^{d } \given S \subseteq [n]}, \]
	where $\vertices(P(G))$ denotes the set of vertices of $P(G)$.
\end{definition}

Let $\lift(P)  = \mset{\sma a \\ \valuationOperator(a) \strix \mid a \in P(G) \cap \ZZ^d} \subseteq \RR^{d+1}$.
We consider the regular subdivision of $P(G)$ that is induced by taking the valuation $\valuationOperator$ as height function, i.e. the projection of the upper convex hull of the lifted polytope. For a detailed description of regular subdivisions, mixed subdivisions and the Cayley trick, we refer the reader to \cite[Section 9.2]{loera_triangulations}.\\

\begin{figure}[h]
	\centering
	\includegraphics[page=5,scale = 0.6]{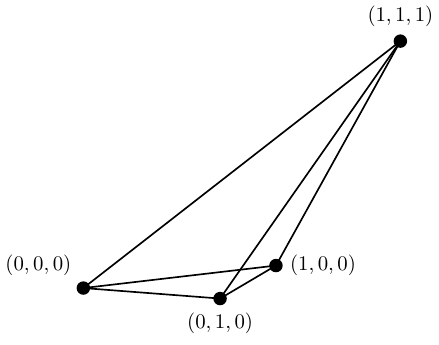} \quad
	\includegraphics[page=6,scale = 0.6]{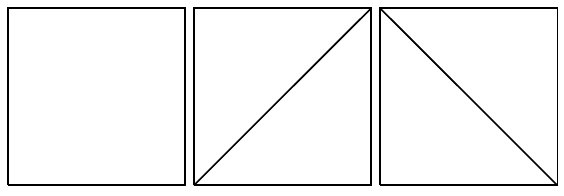}		
	\caption{The polytope $P(K_2)$ (left) and three possible subdivisions of $[0,1]^2$ (right).}
	\label{fig:pk2}
\end{figure}

\begin{example}[Geometry of graphical pricing versus linear pricing]
	
	Consider the case $n=2$. Any valuation $v: \{0,1\}^2 \to \mathbb{R}$ is linearly equivalent to a graphical valuation on $K_2$. Indeed, by the linear transformation $v \mapsto v - v((0,0))$, we can assume that $v((0,0)) = 0$. Then, we can set $w_1 = v((1,0))$, $w_2 = v((0,1))$, and $w_{12} = v((1,1)) - (w_1+w_2)$, so this valuation $v$ is indeed a graphical valuation on $K_2$. In this case, the correlation polytope $P(K_2)$ is a simplex in $\mathbb{R}^3$ (cf. Figure \ref{fig:pk2}).

	Consider an auction for $n = 2$ with two agents. When one works with linear pricing, competitive equilibrium concerns the regular subdivision $\Delta_v$ of $[0,1]^2$. There are three possible subdivisions of $[0,1]^2$ induced by $v$ (cf. Figure \ref{fig:pk2}). In economics, valuations that induce these subdivisions are called linear, complements, and substitutes, respectively. Thus, if $v^1$ is substitutes and $v^2$ is complements, then the set of edges of $\Delta_{v^1}$ and $\Delta_{v^2}$ together do not form a unimodular set: in particular, the vectors $(1,-1)$ and $(1,1)$ are not unimodular, as the determinant of the corresponding $2 \times 2$ matrix is bigger than 1. Therefore, by the Unimodularity Theorem \cite{danilov2004discrete,Baldwin_Klemperer_Understanding_Preferences,Tran_Yu_Product-Mix_Auctions}, competitive equilibrium \emph{can} fail for this auction.
	In contrast, when we work with graphical pricing, competitive equilibrium concerns the regular subdivision $\Delta_v$ of the correlation polytope $P(K_2)$. In this case, \Cref{thm:main} implies that competitive equilibrium \emph{always exists} for graphical pricing, regardless of whether the agents' valuations are complements or substitutes. \\
\end{example}

Let $\widetilde{\valuationOperator}$ be the smallest concave function such that $\widetilde{\valuationOperator}(\bundle) \geq \valuationOperator(\bundle)$ for all $\bundle\in P(G)\cap \ZZ^d$. 
We call $\bundle \in \ZZ^d$ a \emph{lifted point} if $\widetilde{\valuationOperator}(\bundle) = \valuationOperator(\bundle)$, i.e. if $\sma \bundle \\ \valuationOperator(\bundle) \strix$ lies in the upper hull of $\conv(\lift(P))$, as depicted in \Cref{fig:regular_subdiv}. Note that, in particular, a vertex of the regular subdivision is always a lifted point. Further, if the valuation is a linear function, then the induced regular subdivision is trivial. 
Note that 
\[
\demand{\valuationOperator}{-\price} = 
\argmax_{\bundle\in P(G)\cap\ZZ^d} \mset{\valuationOperator(\bundle)+\inner{\bundle,p}} = 
\argmax_{\bundle\in P(G)\cap\ZZ^d} \mset{ \inner{ \sma \bundle \\\valuationOperator(\bundle) \strix, \sma \price \\ 1 \strix } }.
\]
The set $\demand{\valuationOperator}{-\price}$ thus consists of all lifted points of a face of $\lift(P)$ with normal vector $\sma \price \\ 1 \strix$. 
Since a vertex of such a face is always a lifted point, the set of faces of the regular subdivision is given by $\mset{\conv(\demand{\valuationOperator}{-\price})\given \price\in \RR^d}$. 
Hence, a point $\bundle \in \sum_{b=1}^m \demand{\valuation}{\price}$ is contained in a Minkowski sum of faces of $P(G)$. Sums of this form correspond to faces of a mixed regular subdivision of the dilated polytope $mP(G)$, which is defined by the lifting function
\begin{align*}
	\Valuation(\bundle) &= \max \left\{ \sum_{b=1}^m
	\valuation(\bundle^b) \given \sum_{b=1}^m \bundle^b = \bundle,
	\bundle^b \in P(G)\cap\ZZ^{d} \right\} \\
	&=
	\max \left\{ \sum_{b=1}^m
	\inner{\weight,\bundle^b} \given \sum_{b=1}^m \bundle^b = \bundle,
	\bundle^b \in P(G)\cap\ZZ^{d} \right\}.
\end{align*}

\begin{figure}[t]
	\centering
	\includegraphics[page=1]{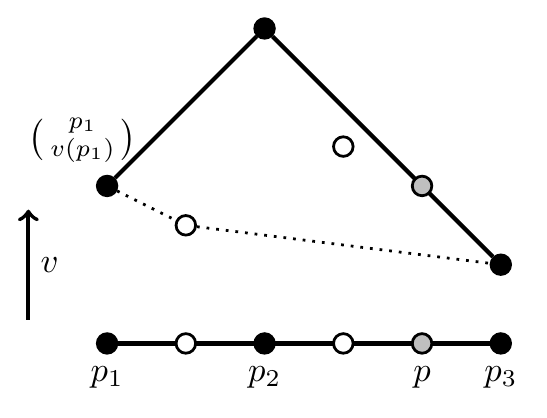}
	\caption{The black points $p_1, p_2$ and $p_3$ are vertices of
		the regular subdivision of the line segment, induced by the
		lifting function $v$. The white points are not lifted
		points. The point $p$ is a lifted point, but not a vertex of the
		regular subdivision: $p$ lies in the interior of the face with
		vertices $p_2$ and $p_3$.} \label{fig:1d.lift}
	\label{fig:regular_subdiv}
\end{figure}

\begin{lemma}\label{adding-const-vectors}
	Let $w\in\RR^d$. The mixed subdivision induced by
	$\Valuation(\bundle)$
	is equal to the subdivision induced by
	$$\Valuation_w(\bundle) = \max \left\{ \sum_{b=1}^m
	\inner{\weight+w,\bundle^b} \given \sum_{b=1}^m \bundle^b = \bundle,
	\bundle^b \in \vertices(P(G)) \right\}.$$ \\
	That is, adding a constant vector $w$ to all weight vectors
	$\weight$ does not change the mixed subdivision of $mP(G)$. 
\end{lemma}

The proof of this lemma can be found in \Cref{sec:proofs}. This lemma implies that, without loss of generality, we can assume that the weights $\weight$ are nonnegative. Note that despite the fact that this does not change the existence of a competitive equilibrium, it does affect the prices at which a competitive equilibrium can be achieved. 

\begin{example}\label{ex:triangle-cont}
	If we add the vector
	$(1,1,1,1,1,1)^T$ to all weights in \Cref{example:triangle}, then a
	pricing equilibrium can be achieved at the price $\price =
	(3,3,1,0,0,0)$ by selling all items to the first agent, in which case
	the seller's revenue is $7$. In accordance to
	\cite[Th. 3.7]{candogan_parrilo_pricing_equilibria} this price vector
	even constitutes a Walrasian equilibrium, as the underlying graph $G=K_3$
	is series-parallel.
\end{example}

\subsection{Competitive equilibrium  and convex geometry}\label{sec:ce.and.geometry}
We now establish the connection between the existence of a competitive equilibrium and properties of the polytope $P(G)$.
A key result is \Cref{prop:key-long}, which gives an if and only if condition for a CE to exist at a given bundle $\bundleStar$ when the agents valuations and the pricing rules are both graphical. Compared to the linear pricing case, each bundle
$\bundleStar$ generates a \emph{set} of points instead of a single
point, and a CE exists \emph{if and only if} \emph{one} of these points
is lifted in a certain regular mixed subdivision. 

\begin{proposition}[CE at an allocation for graphical valuations and pricings]\label{prop:key-long}
	Consider an auction with $m$ agents, anonymous graphical pricing and graphical valuations with a fixed underlying value graph $G$ on $n$ vertices, and $d=n+\vert E(G \vert$.  Let $\bundleStar \in \N^n$ be a fixed bundle and $a\in \piInverse(\bundleStar)$. Then the following are equivalent: 
	\begin{enumerate}[(i)]
		\item For any set of valuations $\{\valuation \mid b \in [m]\}$ there exists an allocation $ (\bundle^1,\dots,\bundle^m)$ and a price $\price\in\RR^d$ such that $a = \sum_{b=1}^m \bundle^b$ and at which a competitive equilibrium exists. 
		\item	For any set of valuations $\{\valuation \mid b \in [m]\}$ there exists a price $\price\in\RR^d$ such that $a \in \sum_{b=1}^m \demand{\valuation}{\price}$. 
		\item\label{item:key-long} For any faces $F^1, \dots,F^m$ of $P(G)$ holds: If $\bundle\in\sum_{b=1}^{m} F^b$ then $\bundle \in \sum_{b=1}^{m} \vertices(F^b)$. \\
	\end{enumerate}
\end{proposition}
The proof of this proposition can be found in \Cref{sec:proofs}. 
Applying the definition of a CE at $\bundleStar$ immediately implies the following:

\begin{corollary}[CE at a bundle for graphical valuations and pricings]\label{cor:key-short}
	Consider an auction with $m$ agents, anonymous graphical pricing and graphical valuations with a fixed underlying value graph $G$ on $n$ vertices, and $d=n+\vert E(G)\vert $.  Let $\bundleStar \in \N^n$ be a fixed bundle. Then the following are equivalent: 
	\begin{enumerate}[(i)]
		\item There exists an $\bundle \in \piInverse(\bundleStar)$ such that for any set of valuations $\{\valuation \mid b \in [m]\}$ there exists an allocation $ (\bundle^1,\dots,\bundle^m)$ and a price $\price\in\RR^d$ such that $a = \sum_{b=1}^m \bundle^b$ and at which a competitive equilibrium exists. 
		\item	There exists an $\bundle \in \piInverse(\bundleStar)$ such that for any set of valuations $\{\valuation \mid b \in [m]\}$ there exists a price $\price\in\RR^d$ such that $a \in \sum_{b=1}^m \demand{\valuation}{\price}$. 
		\item There exists an $\bundle \in \piInverse(\bundleStar)$ such that for any faces $F^1, \dots,F^m$ of $P$ holds: If $\bundle\in\sum_{b=1}^{m} F^b$ then $\bundle \in \sum_{b=1}^{m} \vertices(F^b)$. 
	\end{enumerate}
	
	In particular, if conditions (i)--(iii) hold, then for any set of valuations a competitive equilibrium at $\bundleStar$ exists.
\end{corollary}

\begin{example}[IDP and the nonexistence of CE for $G=K_n$ at every lifted bundle]
	The previous \Cref{prop:key-long} (\ref{item:key-long}) indicates a strong connection to the integer decomposition property of a polytope.
	A $d$-dimensional lattice polytope $P\subset\RR^d$ is said to posses the \emph{Integer Decomposition Property (IDP)} if for every $m\in\N$ and for every lattice point $\bundle\in mP\cap\ZZ^d$ there exist $v_1,\dots, v_m \in P\cap \ZZ^d$ such that $\bundle = \sum_{j=1}^m v_j. $ 
	It is known that the correlation polytope $P(K_n)$ is not IDP for $n\geq 4$. In fact, for $n=4$, $m=4$, the point 
	\begin{equation}\label{eqn:bad.bundle}
		(\bundle_1, \bundle_2, \bundle_3, \bundle_4, \bundle_{12}, \bundle_{13}, \bundle_{14}, \bundle_{23}, \bundle_{24}, \bundle_{34}) = (2,2,2,2,1,1,1,1,1,1)
	\end{equation}
	is the sum of the midpoints of the four edges
	\begingroup
	\allowdisplaybreaks
	\begin{align*}
		\conv\left( (0,0,0,0, 0,0,0,0,0,0), 
		(0,0,0,1, 0,0,0,0,0,0)
		\right), \\	
		\conv\left( (0,1,1,0, 0,0,0,1,0,0), 
		(1,0,1,0, 0,1,0,0,0,0)
		\right), \\
		\conv\left( (0,1,1,1, 0,0,0,1,1,1),
		(1,0,1,1, 0,1,1,0,0,1)	
		\right), \\
		\conv\left( (1,1,0,0, 1,0,0,0,0,0),
		(1,1,0,1, 1,0,1,0,1,0)	
		\right),
	\end{align*}\endgroup
	but cannot be written as the sum of any $4$ lattice points of $P(K_4)$. 
	By \Cref{prop:key-long}, the failure of IDP also implies that there exists some set of valuations $\{v^b\}$ such that competitive equilibrium cannot be achieved at the specific lifted bundle $\bundle$ defined in \eqref{eqn:bad.bundle}. However, \Cref{th:CE} below guarantees that there exists a different lifted bundle $\bundle' \in P \cap \mathbb{Z}^d$ such that $\pi(\bundle) = \pi(\bundle')$, and that competitive equilibrium is achieved at $\bundle'$. 
	This example can be extended to a series of examples for any $n\geq 4$ and $m\geq 4$. 
\end{example}

\section{Main Results} \label{sec:results}
In this section we present and discuss our main results. The proofs follow in the next section.
\subsection{Everyone bids on everything}
Recall that $n$ is the number of distinct items, $m$ is the number of bidders and $\pi$ the projection of a vector with components indexed by $[n]\sqcup E(G)$. In this section, we consider the complete graph $G=K_n$, and vectors are thus of length $d = n+\binom{n}{2}$.
We begin with auctions in which the seller's bundle contains either $0$ or $\rr$ items of each type for a fixed $r\in\N$. Recall that we assume that each agent is only interested in buying at most one item per type.
An important special case is the combinatorial auction where $r=1$. We show that a competitive equilibrium in this scenario can always be achieved: 

\begin{theorem}\label{th:const_size}
	Let $\bundleStar\in \{0,\rr\}^n$ and $m\geq \rr$. If the underlying value graph of all agents is the complete graph, then $\piInverse(\bundleStar) \neq \emptyset$ and for any set of valuations and every $\bundle\in \piInverse(\bundleStar)$ there exists a price $p\in\RR^d$ at which a competitive equilibrium exists.
\end{theorem}

The proof of this theorem gives an explicit construction of how to split the bundle in question. If $r=1$, then $\bundle\in\piInverse(\bundleStar)$ is the characteristic vector of a disjoint union of cliques and the procedure in the proof assigns cliques to agents. A choice of $\bundle\in
\piInverse(\bundleStar)$ corresponds to a choice of connected
components. This construction gives a lot of freedom to the auctioneer: The auctioneer can decide which items are being sold together and is still guaranteed to achieve a competitive equilibrium. The next example shows that even for $r=1$, the existence of a competitive equilibrium can fail when we do not consider the complete graph as value graph.

\begin{example}\label{ex:all-ones-is-not-trivial}
	\begin{figure}[b]
		\centering
		\includegraphics[page=2,scale = 1.25]{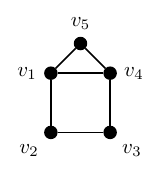}
		\caption{A value graph at which CE may fail at $\bundleStar = (1,1,1,1,1)$.}
		\label{fig:house}
	\end{figure}
	
	Let $G$ be the graph consisting of the cycle $v_1,v_2,v_3,v_4$ together with an additional vertex $v_5$ and edges $v_1v_5, v_4v_5$, as shown in \Cref{fig:house}. Consider the following $4$ edges of the polytope $P(G)$ \\
	\begin{align*}
		F^1 &= \text{conv} \left(
		\begin{pmatrix}
			0, 1, 0, 0, 0, 0, 0, 0, 0, 0, 0
		\end{pmatrix},
		\begin{pmatrix}
			1, 0, 1, 0, 0, 0, 0, 0, 0, 0, 0
		\end{pmatrix} \right) \\
		F^2 &= \text{conv} \left(
		\begin{pmatrix}
			0, 0, 1, 0, 0, 0, 0, 0, 0, 0, 0
		\end{pmatrix},
		\begin{pmatrix}
			0, 1, 0, 1, 0, 0, 0, 0, 0, 0, 0
		\end{pmatrix} \right) \\
		F^3 &= \text{conv} \left(
		\begin{pmatrix}
			0, 0, 0, 0, 1, 0, 0, 0, 0, 0, 0
		\end{pmatrix},
		\begin{pmatrix}
			1, 0, 0, 0, 0, 0, 0, 0, 0, 0, 0
		\end{pmatrix} \right) \\
		F^4 &= \text{conv} \left(
		\begin{pmatrix}
			0, 0, 0, 0, 1, 0, 0, 0, 0, 0, 0
		\end{pmatrix},
		\begin{pmatrix}
			0, 0, 0, 1, 0, 0, 0, 0, 0, 0, 0
		\end{pmatrix} \right) \\
	\end{align*}
	For $\bundleStar=(1,1,1,1,1)$, we have $\piInverse(\bundleStar)\cap\sum_{b=1}^4 F^b = \{(1,1,1,1,1,0,0,0,0,0)\}$ but \\
	$\piInverse(\bundleStar)\cap\sum_{b=1}^4 \vertices(F^b) = \emptyset$. Hence, the assumption of $G=K_n$ in \Cref{th:const_size} is truly necessary. 

\end{example}

Next, we loosen the assumption on $\bundleStar$ and allow
arbitrary $\bundleStar\in\ZZ^n\cap[0,m]^n$. We show that again a CE at
$\bundleStar$ always exists, however we only construct one explicit $\bundle
\in \piInverse(\bundleStar)$ at which a competitive equilibrium is guaranteed to exist.

\begin{theorem}\label{th:CE}
	Let $\bundleStar\in \mset{0,1,\dots,m}^n$. If the underlying graph of
	all valuations is the complete graph, then there exists an $a\in\piInverse(\bundleStar)$ such that for any set of valuations there exists a price $p\in\RR^d$ at which a competitive equilibrium exists.
\end{theorem}

The procedure in which the bundle  can be split up to achieve a competitive equilibrium is as follows: \\
If $\bundleStar \in \{0,1\}^n$, then auctioneer sells the entire bundle to one agent, i.e. there is one agent who gets an item of each type $i\in[n]$ such that $\bundleStar_{i} = 1$. If $\bundleStar \in \{0,1,2\}^n$, then the auctioneer sells the items in two bundles: There is one agent who will be offered one item of each type where $\bundleStar_i > 0$. A second agent will be made an offer for all remaining items, i.e. all items such that $\bundleStar_i = 2$. And so on. \\

Even though this result might seem underwhelming in the case of combinatorial auctions, in which \Cref{th:const_size} gives the auctioneer much more flexibility, it still provides an answer regarding the theoretical question of existence of competitive equilibrium for arbitrary bundles.

\subsection{Everything is bid on by someone}

We can relax the condition on the valuations by allowing weight vectors $\weightOperator\in \RR\cup\{-\infty\}$. We assume that for every agent $b \in [m]$ the
valuation function $\valuation: (\RR\cup\{-\infty\})^d \to
\RR\cup\{-\infty\}$ is of the form $\valuation(\bundle) =
\inner{\weight,\bundle}$ such that $\weight$ has finite value on the
vertices and edges of some clique of $K_n$ and has weight $-\infty$ on
all vertices outside the clique.
If the vertex set of this clique is the subset $\partition^b\subseteq
[n]$, we say that agent $b$ \emph{bids on the subgraph}
$K_{\partition^b}$.  The \emph{support} of a valuation $\valuation$ is the set of vertices and edges where $\weight$ has finite value.
A set of valuations $\mset{\valuation \mid b\in[m]}$ is \emph{covering} if every agent $b \in [m]$ bids on
a clique $K_{\partition^b}$ such that $\bigcup_{b \in [m]}
\partition^b = [n]$.
A vector $\bundle\in\RR^d$ is \emph{compatible} with this covering if for every $ij \in \binom{[n]}{2}$ such that $a_{ij}> 0$ there is a set $S^b$ such that $i,j\in S^b$.
We are now ready to state the generalization of \Cref{thm:main}.
\begin{theorem}\label{thm:main-general}
	If $\mathcal{V}$ is the collection of sets of covering graphical valuations, then a
	competitive equilibrium with anonymous graphical pricing is guaranteed
	to exist at \emph{any} bundle $a^\ast \in \{0,1\}^n$ for
	\emph{any} auction with valuations in $\mathcal{V}$. 
\end{theorem}

The above theorem immediately follows from this more technical theorem:
\begin{theorem}\label{thm:generalize-to-main}
	Let $\bundleStar\in \{0,\rr\}^n$ and suppose that every agent $b \in [m]$ bids on
	a clique $K_{\partition^b}$ such that $\bigcup_{b \in [m]}
	\partition^b = [n]$. Then for any set of valuations $\{\valuation \mid \valuation \text{ is supported on } K_{\partition^b} \}$ and any  $\bundle\in
	\piInverse(\bundleStar)$ that is compatible with the covering there exists a price $p\in\RR^d$ at which a competitive equilibrium
	exists.
\end{theorem}

The proof of \Cref{thm:generalize-to-main} implies \Cref{alg:ce} below to compute an optimal competitive equilibrium.


\begin{algorithm}[H]
	\caption{} \label{alg:ce} 
	\begin{algorithmic}[1]
		\Require The agent's quadratic bids $\weight[1],\dots,\weight[m]$.
		\Ensure An allocation-price pair $((\bundle^1,\dots,\bundle^m), \price)$ which achieves an optimal competitive equilibrium.
		\State Check that $\weight[1],\dots,\weight[m]$ give a covering of $[n]$
		\State Choose a partition $(V^1,\dots,V^m)$ of $[n]$ such that $\weight$ is supported on $V^b$ (this defines $\bundle \in \piInverse(\bundleStar)$)
		\State Let $\weighttilde$ be the vector where every weight of $-\infty$ in $\weight$ is replaced by $-M$ for some large $M>0$. Go through all decompositions $\bundle = \sum_{b=1}^m \bundle^{ b} ,  \bundle^{b} \in\vertices(P(K_n))$, check for feasibility and compute the value of
		\begin{align*}
			&\max \inner{\price,\bundle} \\
			&\text{s.t.} \inner{\price - \weighttilde, \bundle^{b}} \geq \inner{\price - \weighttilde, \bundle^\prime} \forall \bundle^\prime \in \vertices(P(K_n)), b\in [m] 
		\end{align*}
		\State \Return $((\bundle^1,\dots,\bundle^m), \price)$ at which the linear program above has the maximal value.
	\end{algorithmic}
\end{algorithm}

\section{Proofs}\label{sec:proofs}

We now proof the results that are stated in the previous sections. First, we give the proofs of the results stated in \Cref{sec:background}. Next we proof auxiliary results that are needed for the proofs of the main results, which are stated in \Cref{sec:results}. Finally, we give an index that refers the reader to the necessary auxiliary results for the respective main result.

\subsection{Convex results}

\begin{proof}[Proof of \Cref{adding-const-vectors}:]
	The lifting function on the Cayley polytope $\mathcal{C}(P(G),\dots,P(G))\subseteq~\RR^d \times~\RR^m$ corresponding to the mixed subdivision on $mP$ induced by $\Valuation$ is 
	\[
	\widetilde{\Valuation} : (\bundle, e_b) \mapsto \valuation(\bundle) = \inner{\weight,\bundle}
	\]
	(cf. \cite[Cor.~4.10]{joswig}). Adding the constant vector $w$ to all weight vectors $\weight$ yields another lifting function of the Cayley polytope
	\[
	\widetilde{\Valuation}_w : (\bundle,e_b) \mapsto \inner{\weight + w, \bundle} = \valuation(\bundle) + \inner{\begin{pmatrix} w \\ \mathbf 0 \end{pmatrix}, \begin{pmatrix} \bundle \\ e_b \end{pmatrix}}.
	\]
	Thus, $\widetilde{\Valuation}_w = \widetilde{\Valuation} + \inner{\begin{pmatrix} w \\ \mathbf 0 \end{pmatrix}, \cdot}$, i.e. adding $w$ to all weights $\weight$ amounts to  adding a linear functional to the lifting function $\Valuation$. This operation does not change the regular subdivision on $\mathcal{C}(P(G),\dots,P(G))$ and hence leaves the corresponding mixed subdivision of $mP(G)$ unchanged.
\end{proof}
$ $
\begin{proof}[Proof of \Cref{prop:key-long}:]
	First note, that (i) and (ii) are equivalent by the definition of CE at $\bundle \in \piInverse(\bundleStar)$ in \Cref{def:ce}. It thus remains to show the equivalence of (ii) and (iii). 
	Before that, note that each face $F^b$ of the regular subdivision of
	$P(G)$ induced by a valuation  $\valuation$ is given by $F^b =
	\conv(\demand{\valuation}{\price^b})$ for some price $\price^b \in
	\RR^d$.
	The valuations $\valuation$ are linear functions on $\RR^d$. Therefore, the
	regular subdivision induced by $\valuation$ on $P(G)$ is trivial.
	The set of lifted points of $P(G)$ is the set of vertices
	$\vertices(P(G))$. Further, each face $F$ of the regular subdivision
	of $mP(G)$ induced by the aggregate valuation $\Valuation$ corresponds to a
	set $\demand{\Valuation}{\price}$ for some $\price\in\RR^d$, where
	$\demand{\Valuation}{\price}$ is the set of all lifted points in
	$F$. Since $\demand{\Valuation}{\price} = \sum_{j=1}^{m}
	\demand{\valuation}{\price}$, the set of lifted points of $mP(G)$ is
	\[\bigcup_{\price\in\RR^d} \demand{\Valuation}{\price} =
	\bigcup_{\price\in\RR^d} (\sum_{j=1}^m
	\demand{\valuation}{\price}) \subseteq \sum_{b=1}^m
	\vertices(P(G)).\] 
	We now show $\neg (ii) \implies \neg (iii)$. Suppose there is a set of valuations $\mset{\valuation \given b\in[m]
	}$ such that for all $\price\in\RR^d$ holds
	$\bundle\not\in \demand{\Valuation}{\price} = \sum_{b=1}^m
	\demand{\valuation}{\price}$. Since $\bundle = \piInverse(\bundleStar) \subseteq mP(G)$, $\bundle$ lies in some face of the regular subdivision of
	$mP(G)$ induced by $\Valuation$. These faces are in bijection with
	the distinct sets $\demand{\Valuation}{\price}$, so there exists
	some $\price\in\RR^d$ such that $\bundle\in
	\conv(\demand{\Valuation}{\price})$. The assumption $\bundle\not\in
	\demand{\Valuation}{\price}$ implies that $\bundle$ is not a lifted
	point. Note that, since Minkowski summation and the operator of
	forming convex hulls commute,
	
	\begin{align*}
		\bundle\in \conv(\demand{\Valuation}{\price})
		&=
		\conv(\sum_{b=1}^m
		\demand{\valuation}{\price})
		\\ 
		&= \sum_{b=1}^m \conv(\demand{\valuation}{\price}) \\
		&= \sum_{b=1}^m F^b
	\end{align*}
	for some faces $F^1,\dots,F^m$ of $P(G)$. Since $P(G)$ is a
	$0/1$-polytope, for each $b\in [m]$ the vertices $\vertices(F^b)$
	are precisely the lifted points of $F^b$ in the trivial subdivision
	induced by~$\valuation$, i.e. $\vertices(F^b) =
	\demand{\valuation}{\price}$. Therefore \[\sum_{b=1}^m
	\vertices(F^b) =\sum_{b=1}^m \demand{\valuation}{\price} =
	\demand{\Valuation}{\price},\] so $\sum_{b=1}^m \vertices(F^b)$ is
	a set of lifted points in the subdivision induced by
	$\Valuation$. By assumption, $\bundle$ is not a lifted point, and
	thus $\bundle\not\in \sum_{b=1}^m \vertices(F^b)$. \\
	Next, we show $\neg (iii) \implies \neg (ii)$. Suppose $\bundle\in \sum_{b=1}^m F^b$ but
	$\bundle\not\in\sum_{b=1}^m (\vertices(F^b))$ and let $\price\in\RR^d$. 
	Let $-p^b$ be the normal vector the face $F^b$ and consider the valuation 
	\[
	\valuation(\bundle) = \inner{\price - \price^b, \bundle}.
	\]
	Then 
	\[
	\demand{\valuation}{\price} = \argmax_{\bundle\in P(G)}\mset{
		\inner{\price -
			\price^b,\bundle}-\inner{\price,\bundle}}
	=  \argmax_{\bundle\in P(G)}\mset{
		-\inner{\price^b,\bundle}} = \vertices(F^b)
	\]
	and hence $\bundle\not\in\sum_{b=1}^m (\vertices(F^b))$  implies $\bundle\not\in\sum_{b=1}^m \demand{\valuation}{\price}$.
	
\end{proof} $ $

\subsection{Auxiliary results}

For the complete graph $G=K_n$, the polytope $P(K_n)$ is generally known as the \emph{correlation polytope}  or \emph{boolean quadric polytope} \cite{Padberg89}. It is also affinely equivalent to the well-known cut-polytope \cite{lectures}. It has been widely studied, yet, the hyperplane description of this polytope remains unknown. We thus make use of a linear relaxation given in \cite{Padberg89}:

\begin{enumerate}[(i)]
	\item $x_{ij}\geq0$
	\item $x_i - x_{ij} \geq 0$
	\item $x_i + x_j - x_{ij} \leq m$
	\item $x_i + x_{jk} - x_{ij} - x_{ik} \geq 0$
	\item $x_i + x_j + x_k - x_{ij} - x_{ik} - x_{jk} \leq m$
\end{enumerate} 
for all $i,j,k\in [n]$. Note that the inequalities (i) -- (v)  define a bounded polyhedron, containing the polytope $mP(K_n)$.
In fact, for $m=1$ the inequalities (i) -- (iii) together with the constraint $x\in\ZZ^d$ yield the vertices of the correlation polytope, and for $n\leq 3$ the inequalities (i) -- (v) suffice to describe the polytope completely \cite[Section 2]{Padberg89}. \\

\begin{lemma}\label{lem:decomposition_into_graphs}
	Let $\bundle\in \{0,\rr\}^d$ satisfy inequalities (i), (ii) and (iv) for some $\rr\in\N$. Then $\bundle$ is the sum of characteristic vectors of pairwise disjoint complete graphs $G_1\dots G_s$. More precisely,
	\[\bundle = \sum_{t=1}^{s} \rr \chi^{t}, \]
	where $\chi^{t}$ denotes the characteristic vector of $G_t$ and $s\leq n$.
\end{lemma}

\begin{figure}[ht]
	\centering
	\includegraphics[page=3]{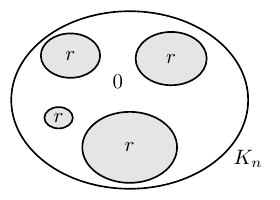}
	\caption{The vector $\bundle$ is the sum of $\rr$ copies of characteristic vectors of pairwise disjoint complete graphs $G_1,\dots,G_s$.}
\end{figure}

\begin{proof}
	Since $\bundle$ satisfies inequalities (i) and (ii), $\frac{1}{\rr}\bundle\in \{0,1\}^d$ satisfies these inequalities as well. We can therefore interpret $\frac{1}{\rr}\bundle$ as the characteristic vector of a (non-complete) graph $\hat{G}$, i.e $\frac{1}{\rr}\bundle=\chi_{\hat{G}}$ and thus $a=\rr\chi_{\hat{G}}$ where $V(\hat{G})\subseteq [n]$.
	
	Inequality (iv) forbids an induced subgraph consisting of a path of length two (and not a triangle), since otherwise this would imply a set $\{i,j,k\}$ of indices such that 
	\[\bundle_i=\bundle_j=\bundle_k=\bundle_{ij}=\bundle_{ik}=\rr, \quad \bundle_{jk}=0,\] and thus
	\[ \bundle_i + \bundle_{jk} - \bundle_{ij} - \bundle_{ik} = -\rr. \]
	Whenever two vertices are connected via a path of length two, they must hence be adjacent. Inductively follows, that whenever two vertices are connected via a path of arbitrary length, they must be adjacent as well. Therefore, $\hat{G}$ is the disjoint union of non-empty complete graphs $G_1,\dots G_s$ and 
	\[ \bundle = \sum_{t=1}^{s} \rr \chi^{t}, \]
	where $\chi^t$ denotes the characteristic vector of $G_t$. Further,
	\[ 
	n \geq  \vert V(\hat{G}) \vert
	= \sum_{t=1}^{s} \vert V(G_t) \vert
	\geq \sum_{t=1}^{s} 1
	= s.
	\]
\end{proof}

\begin{proposition}\label{prop:decomp:const}
	Let $\bundle \in \mset{0,\rr}^d$. Then for all faces $F^1,\dots,F^m$ of $P(K_n)$ containing $\bundle\in\sum_{b=1}^m F^b$ in their Minkowski sum, there exist vertices $\chi^b$ of $F^b$ such that $\bundle=\sum_{b=1}^m \chi^b$. 
\end{proposition}

\begin{proof}
	
	If there exist faces $F^1,\dots,F^m$ of $P(K_n)$ such that $\bundle\in\sum_{b=1}^m F^b$, then in particular $a\in mP(K_n)$. Thus, the inequalities (i) -- (v) hold for $\bundle$ and by \Cref{lem:decomposition_into_graphs} we have
	\[\bundle = \sum_{t=1}^{s} \rr \chi^{t}, \]
	where $\chi^t$ is the characteristic vector of a complete graph $G_t$ for each $t\in[s]$, the graphs $G^1,\dots G^s$ are pairwise disjoint, and $s\leq n$. Choose $ \mset{k_i \given i\in[s]} $ such that $k_i \in V(G_i)$ for each $i\in[s]$ and
	consider the following inequalities 
	\[ 
	\text{(vi') } \
	\sum_{i\in [s]} x_{k_i} 
	- \sum_{ij\in \binom{[s]}{2}} x_{k_i k_j}
	\leq 1
	\qquad \
	\text{(vi) } \
	\sum_{i\in [s]} x_{k_i} 
	- \sum_{ij\in \binom{[s]}{2}} x_{k_i k_j}
	\leq m.
	\]
	Each vertex of $P(K_n)$ satisfies the inequality (vi') and so $\bundle\in mP(K_n)$ satisfies (vi).
	Since $\bundle_{k_i}=\rr$ and $\bundle_{k_i k_j}=0$ for all $i,j\in[s]$, we have
	
	\[ 
	\sum_{i\in [s]} \bundle_{k_i} 
	- \sum_{ij\in \binom{[s]}{2}} \bundle_{k_i k_j}
	= s\cdot \rr - 0	
	\leq m.
	\]
	We can thus write $ \bundle = \sum_{b=1}^{m} \hat{\chi}^{b}, $ where in this representation we have $\rr$ copies of $\chi^t$ for each $t\in [s],$ and $m-s\cdot \rr$ copies of the characteristic vector $\chi^0=\mathbf{0}$ of the empty graph $G_0$. 
	Our task is now to distribute the vertices $\hat{\chi}^b$ such that $\hat{\chi}^b \in F^b, b\in [m]$.
	
	By assumption,
	$\bundle\in (F^1 + \dots + F^m) $, i.e. we can find a subset $V(F^\letter)$ of vertices of the face $F^\letter$
	such that $\lambda_\point > 0$ for all $\point\in V(F^\letter)$ and
	\[\bundle = \sum_{\letter=1}^m \sum_{\point\in V(F^\letter)} \lambda_\point \point, 
	\quad
	\sum_{\point\in V(F^\letter)} \lambda_\point = 1  . \]
	We now show
	that $\bigcup_{\letter=1}^m V(F^\letter) = \left\{ \hat{\chi}^{b} \given b\in
	[m] \right\}$ and deduce from this that 
	\[
	r\chi^t = \sum_{b=1}^{m} \sum_{\substack{q \in V(F^b) \\ q = \chi^t}} \lambda_q q.
	\]  Then the following \Cref{lem:perfect:matching} implies that
	there is a labeling of the faces such that $\hat{\chi}^b \in F^b$ for
	all $b\in [m]$.
	\\ 
	
	We begin with showing $\bigcup_{\letter=1}^m V(F^\letter) \subseteq \left\{ \hat{\chi}^{b} \given b\in
	[m] \right\}$.
	Let $ \hat{\point}\in V(F^\letter) $ for some $ \letter\in[m] $ such that $ \hat{\point}\neq \mathbf{0}. $ Then there is some $i\in[n]$ such that $\hat{\point}_i = 1$ and since $ \lambda_{\hat{\point}}>0 $ we have $ \bundle_i = \rr. $ Hence, there is some $ t\in[s] $ such that $ i\in V(G_t). $
	We now show that $ \hat{\point}=\chi^t. $\\
	If $\hat{\point}_j = 1 $, then $ \hat{\point}_{ij}=1 $ and therefore $\bundle_{ij}=\rr$. This implies that $i$ and $j$ are contained in the same connected component in $\hat{G}$, i.e. $j\in V(G_t)$. Note that if $i$ is an isolated vertex in $\hat{G}$, then the assumption $\hat{\point}_j=1$ immediately leads to a contradiction, so in this case we have indeed $\hat{\point}=e_i$.\\
	If $\hat{\point}_j = 0$, then $ 0 = \hat{\point}_{ij} < \hat{\point}_i. $ This implies $ \bundle_{ij} = 0, $ since otherwise (by (ii))
	\begin{equation*}
		\rr = \bundle_{ij}
		= \sum_{\letter=1}^m \sum_{\point\in V(F^\letter)} \lambda_\point \point_{ij}
		< \sum_{\letter=1}^m \sum_{\price\in V(F^\letter)} \lambda_\point \point_i
		= \bundle_i = \rr. 
	\end{equation*}
	Hence, $ j\not\in V(G_t) $ and so $\hat{\point}=\chi^t \in \left\{ \hat{\chi}^{b} \given b\in
	[m] \right\}.$ \\
	
	Next, we show that $\bigcup_{\letter=1}^m V(F^\letter) \supseteq \left\{ \hat{\chi}^{b} \given b\in
	[m] \right\}$. Let $t\in [s]$ and $i\in V(G_t)$. Then $\bundle_i = \rr$,
	so there exists some $\hat{\point}\in \bigcup_{\letter=1}^m V(F^\letter)$
	such that $\hat{\point}_i = 1$. By the above, this implies
	$\hat{\point}=\chi^t$, and hence $\left\{\chi^1,\dots,\chi^s \right\}
	\subseteq \bigcup_{\letter=1}^m V(F^\letter)$. It remains to show that $\mathbf 0 \in \bigcup_{\letter=1}^m V(F^\letter)$ if and only if $s\cdot r\leq m$.
	For $t=0,\dots,s$, we define
	\[
	\mu_t = \sum_{\letter=1}^{m} \sum_{ \substack{ \point\in V(F^\letter) \\ \point=\chi^t } } \lambda_\point.
	\]
	We show that $\mu_0 = m - s\cdot r$. This implies that $\mathbf{0}\in
	\bigcup_{\letter=1}^m V(F^\letter)$ if and only if $ s\cdot \rr < m, $
	which is equivalent to $\mathbf{0}  \in \left\{\hat{\chi}^{b}
	\given j\in [m] \right\}.$ We can write 
	\[\bundle = \sum_{\letter=1}^m \sum_{\point\in V(F^\letter)} \lambda_\point \point
	\ = \sum_{t=0}^{s} \mu_t \chi^t.
	\]
	For $i\in V(G_t)$ we have
	\[\rr = \bundle_i = \sum_{t=0}^{s} \mu_t \chi^t_i 
	= \mu_t
	\]
	and so 
	\[ 
	s\cdot \rr + \mu_0 = \sum_{t=0}^{s} \mu_t 
	= \sum_{t=0}^{s} \sum_{\letter=1}^{m} \sum_{ \substack{ \point\in V(F^\letter) \\ \point=\chi^t } } \lambda_\point
	= \sum_{\letter=1}^{m} \sum_{ \point\in V(F^\letter) } \lambda_\point
	= \sum_{\letter=1}^{m} 1 = m.
	\]
	Thus, $\bigcup_{\letter=1}^m V(F^\letter) = \left\{ \hat{\chi}^{b}
	\given j\in [m] \right\}$. \\
	
	We finish the proof by showing that	
	\[
	r\chi^t = \sum_{b=1}^{m} \sum_{\substack{q \in V(F^b) \\ q = \chi^t}} \lambda_q q
	\]
	for all $t\in[s]$.
	Let $i\in V(G_t)$ for some $t\in[s]$. Since $\bigcup_{\letter=1}^m V(F^\letter) = \left\{ \hat{\chi}^{b}
	\given j\in [m] \right\}$ and the graphs $G^1,\dots,G^s$ are pairwise disjoint, we have that 
	
	\[
	r\chi^t_i = r = a_i = \sum_{b=1}^{m} \sum_{q \in V(F^b)} \lambda_q q_i = \sum_{b=1}^{m} \sum_{\substack{q \in V(F^b) \\ q = \chi^t}} \lambda_q q_i.
	\]
	A similar statement holds for the edge $ij$ for $i,j\in V(G^t)$. If $i\not\in V(G^t)$ then 
	\[
	r\chi^t_i = 0 = \sum_{b=1}^{m} \sum_{\substack{q \in V(F^b) \\ q = \chi^t}} \lambda_q q_i.
	\]
	and similarly for any $i,j\in[n]$ such that $ij$ is not an edge in $G^t$. Hence, all assumptions of 	\Cref{lem:perfect:matching} are fulfilled and 
	\Cref{lem:perfect:matching} implies that
	there is a labeling of the faces such that $\hat{\chi}^b \in F^b$ for
	all $j\in [m]$.
\end{proof}

\begin{lemma}\label{lem:perfect:matching}
	Suppose $\bundle = \sum_{t=1}^s \mu_t \chi^{t} = \sum_{b=1}^m \hat{\chi}^b$, such that 
	\begin{enumerate}[$\bullet$]
		\item 	$\mset{\chi^t \mid t\in[s]} =\mset{\hat{\chi}^b\mid b\in [m]}$,
		\item  $\chi^{t}$ is the
		characteristic vector of a complete subgraph $G_t$ of
		$K_n$, 
		\item  $\mu_t \in \N$ such that $\sum_{t=1}^{s} \mu_t = m$.
	\end{enumerate}
	Further, let $\bundle$ be contained in the Minkowski sum of
	faces $F^b, b\in[m]$, i.e.
	\[\bundle = \sum_{b=1}^m \sum_{\point\in V(F^b)} \lambda_\point
	\point, \sum_{\point\in V(F^b)} \lambda_\point = 1 \]
	where $V(F^b)$ is a subset of vertices of the face $F^b$ and
	$\lambda_\point > 0$.
	If for every $t\in[s]$ holds
	\[
	\mu_t \chi^t = \sum_{b=1}^m \sum_{\substack{\point\in V(F^b) \\ \point = \chi^t}} \lambda_\point
	\point,
	\]
	then there is a labeling of the faces
	such that $\chi^b \in F^b$.
\end{lemma}

\begin{proof}
	First note, that we can write 
	\[
	\mu_t \chi^t = \sum_{b=1}^m \sum_{\substack{\point\in V(F^b) \\ \point = \chi^t}} \lambda_\point
	\point =  \sum_{b=1}^m \lambda_t^b \chi^t,
	\]
	where 
	\[
	\lambda_t^b =   \sum_{\substack{\point\in V(F^b) \\ \point = \chi^t}} \lambda_\point
	\]
	and thus $0\leq\lambda_t^b \leq1$, and $\lambda_t^b > 0$ if and only if $\chi^t \in V(F^b)$. Further, note that this implies 
	\[
	\mu_t = \sum_{b=1}^m \lambda_t^b \ \text{  and   } \ \sum_{t=1}^s \lambda_t^b = \sum_{\point\in V(F^b)} \lambda_\point = 1 \text{ for a fixed } b\in[m].
	\]
	We now show that this constitutes a flow in a network. The Max-flow-min-cut theorem (e.g. \cite[Th. 7.2]{schrijver_theorylinearinteger}) then implies that an integer flow exists, which will give us the desired assignment of vertices $\chi^t$ and faces $F^b$. 
	
	Let $D=(V,E)$ be a directed graph with vertices 
	\[
	V = \mset{F^b \mid b\in [m]} \cup \mset{\chi^t \mid t\in[s]} \cup \mset{\hat{v}, \hat{w}},
	\]
	where $\hat{v}$ is the vertex of $D$ representing the source of the network, and $\hat{w}$ represents the sink. The directed edges of $D$ are
	\begin{align*}
		E = &\mset{(\hat{v}, F^b) \mid b \in [m]} \cup \mset{ (F^b, \chi^t) \mid b\in [m], t\in[s], \chi^t \in V(F^b) } \\ &\cup \mset{ (\chi^t, \hat{w}) \mid t\in[s]},
	\end{align*}
	and we consider the capacity function $c: E \to \RR_{>0}$ such that 
	\[
	c(e) = \begin{cases}
		1 & e = (\hat{v}, F^b) \text{ for some } b\in[m], \\
		1 & e = (F^b, \chi^t)  \text{ for some } b\in[m], t\in[s], \\
		\mu_t & e = (\chi^t, \hat{w})  \text{ for some } t\in[s].
	\end{cases}
	\]
	We claim that the following function $f:E\to \RR_{>0}$ constitutes a maximal flow on $D$
	\[
	f(e) = \begin{cases}
		1 & e = (\hat{v}, F^b) \text{ for some } b\in[m], \\
		\lambda_t^b & e = (F^b, \chi^t)  \text{ for some } b\in[m], t\in[s], \\
		\mu_t & e = (\chi^t, \hat{w})  \text{ for some } t\in[s],
	\end{cases}
	\]
	i.e. that for all vertices $v \in V\setminus\mset{\hat{v}, \hat{w}}$ holds
	\[
	\sum_{\substack{w \in V \\ e = (v,w)}} f(e) = \sum_{\substack{w \in V \\ e = (w,v)}} f(e) .
	\]
	To see this, let first $v = F^b$ for some $b\in [m]$. Then
	\[
	\sum_{\substack{w \in V \\ e = (w,v)}} f(e) = f(\hat{v},F^b) = 1  \text{  and  } 
	\sum_{\substack{w \in V \\ e = (v,w)}} f(e) = \sum_{t=1}^s \lambda_t^b = 1.
	\]
	If $v = \chi^t$ for some $t\in[s]$, then 
	\[
	\sum_{\substack{w \in V \\ e = (w,v)}} f(e) = \sum_{b=1}^m \lambda_t^b = \mu_t \ \text{ and } \ 	\sum_{\substack{w \in V \\ e = (v,w)}} f(e) = f(\chi^t, \hat{w}) = \mu_t.
	\]
	This is indeed a flow of maximal size $m$, since 
	$$m = \sum_{t=1}^s \mu_t = \sum_{t=1}^s f(\chi^t, \hat{w}) = \sum_{t=1}^s c(\chi^t, \hat{w}),$$ 
	and the flow cannot exceed the capacity. Hence, there exists an integral flow $f'$ of size $m$. This is a flow such that for each $b\in[m]$ there exists exactly one $t\in [s]$ such that  $f'(F^b, \chi^t) = 1$, and has value $0$ otherwise. On the other hand,  for each $t\in[s]$ there are exactly $\mu_t$ facets $F^b, b\in [m]$ such that $f'(F^b, \chi^t) = 1$. Since for each $t\in [s]$ there are exactly $\mu_t$ copies of $\chi^t$ in the representation $a = \sum_{b=1}^m \hat{\chi}^b$, there is a labeling of $\hat{\chi^1},\dots,\hat{\chi^m}$ such that $f'(F^b,\hat{\chi}^b)=1$ for all $b\in[m]$. This induces the desired assignment $\mset{(F^b, \hat{\chi}^t) \mid f'(F^b, \chi^t)=1}$.
\end{proof}

\begin{proposition}\label{prop:CE_always_exists}
	Let $\bundleStar\in \mset{0,1,\dots,m}^n$. Then there exists a point
	$\bundle\in \piInverse(\bundleStar)\cap mP(K_n)$ such that the following holds:
	For all faces $F^1,\dots,F^m$ of $P(K_n)$ containing
	$\bundle\in\sum_{b=1}^m F^b$ in their Minkowski sum, there exist
	vertices $\chi^b$ of $F^b$ such that $\bundle=\sum_{b=1}^m \chi^b$.
\end{proposition}

\begin{proof}
	We define $\bundle$ as follows:
	\begin{equation*}
		\bundle_i= \bundleStar_i
		\qquad
		\bundle_{ij} = \min \left\{ \bundleStar_i, \bundleStar_j \right\}.
	\end{equation*}
	We show that $\bundle$ is the sum of characteristic vectors of
	complete graphs $G_{m}\subseteq\dots\subseteq G_{1}$. Let $t_0 = 0$
	and $t_1<t_2<\dots<t_s$ be the distinct non-zero values of
	$\bundleStar,$ i.e.	$\mset{\bundleStar_i\given
		i\in[n]}\setminus\{0\}=\mset{t_l\given l\in[s]}$. Let $\chi^{t_l}$
	be given by
	\begin{equation*}
		\chi^{t_l}_i= 
		\begin{cases}
			1, & \text{if}\ \bundle_i\geq t_l \\ 
			0, & \text{otherwise}
		\end{cases}\ \text{ for } i\in [n],
		\qquad
		\chi^{t_l}_{ij} = \chi^{t_l}_i \chi^{t_l}_j\ \text{ for } ij\in \binom{[n]}{2}.
	\end{equation*}
	Then $\chi^{t_l}$ is the characteristic vector of a complete graph
	$G_{t_l}$ and for a non-zero $\bundle_i = t_k$ we have
	\[
	a_i = t_k = \sum_{l=1}^k (t_l - t_{l-1} )
	= \sum_{l=1}^k (t_l - t_{l-1})\chi^{t_l}_i = \sum_{l=1}^s (t_l - t_{l-1}) \chi^{t_l}_i.
	\]
	An analogous statement holds for $\bundle_{ij} =\min\{a_i,a_j\}$ and thus
	\[
	\bundle = \sum_{l=1}^s (t_l - t_{l-1}) \chi^{t_l}
	\]
	for complete graphs $G_{t_s}\subsetneq \dots \subsetneq G_{t_1}$.
	Taking $t_l - t_{l-1}$ copies of $G_{t_l}$ and $m-t_s$ copies of the empty graph, we can write
	
	\[
	\bundle = \sum_{b=1}^m \hat{\chi}^{b}
	\]
	for graphs $G_m\subseteq \dots \subseteq G_1$, where $G_b = G_{t_l}$
	for any $b \in [m]$ such that $t_{l-1} + 1  \leq b \leq t_l $ and
	$\hat{\chi}^\letter$ denotes the characteristic vector of $G_\letter$. Thus, $\bundle$
	is the sum of $n$ vertices of $P(K_n)$ and is therefore contained in
	$\piInverse(\bundleStar)\cap mP(K_n)$. \\

	\begin{figure}[h]
		\centering
		\includegraphics[page=4]{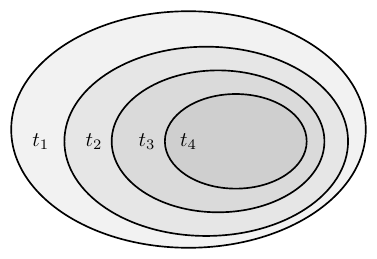}
		\caption{The vector $\bundle$ is the sum of characteristic vectors
			of nested complete graphs $G_{t_s}\subsetneq \dots \subsetneq
			G_{t_1}$.} 
		\label{fig:ex_2_3}
	\end{figure}
	
	Suppose $\bundle\in (F^1 + \dots + F^m) $, i.e.
	\[
	\bundle = \sum_{\letter=1}^m \sum_{\point\in V(F^\letter)}
	\lambda_\point \point,\quad
	\sum_{\point\in V(F^\letter)} \lambda_\point = 1  ,
	\]
	where $V(F^\letter)$ is a subset of vertices of the face
	$F^\letter$ and $\lambda_\point > 0$ for all $\point\in
	V(F^\letter)$. Note that all vertices of the polytope $P(K_n)$
	satisfy inequalities (i)--(v) for $m=1$. Similar to the proof of \Cref{prop:decomp:const}, we now show
	that $\bigcup_{\letter=1}^m V(F^\letter) = \left\{ \chi^{b} \given b\in
	[m] \right\}$ and deduce from this that 
	\[
	(t_l - t_{l-1}) \chi^{t_l} = \sum_{b=1}^{m} \sum_{\substack{q \in V(F^b) \\ q = \chi^{t_l}}} \lambda_q q.
	\]  Then \Cref{lem:perfect:matching} implies that
	there is a labeling of the faces such that $\hat{\chi}^b \in F^b$ for
	all $b\in [m]$. \\
	
	We begin with $\bigcup_{\letter=1}^m V(F^\letter) \subseteq \left\{ \hat{\chi}^{b} \given b\in
	[m] \right\}$. Let $\hat{\point}\in V(F^\letter)$ be a non-zero vector and $i\in[n]$
	such that $\bundle_i = \min_{j\in [n]} \mset{\bundle_j \given
		\hat{\point}_j = 1}$. Note that since $\lambda_{\hat{\point}}>0$ and
	$\hat{\point}_i =1$, we have $\bundle_i>0$, so $\bundle_i = t_k$ for
	some non-zero value.
	We show that $\hat{\point}=\chi^{t_k}$.
	Let $j\in [n]$. If $\bundle_j < t_k$, then $\hat{\point}_j = 0$ by the
	choice of $i$ and thus $\hat{\point}_{ij} = 0$ by (ii). If
	$\bundle_j\geq t_k$, then $\bundle_{ij} = \bundle_i = t_k$ by the
	definition of $\bundle$ and so $\hat{\point}_{ij}=\hat{\point}_i=1$,
	since otherwise (by (ii))
	\begin{equation}\label[ineq]{eq:x_implies_p}
		t_k = \bundle_{ij}
		= \sum_{\letter=1}^m \sum_{\point\in V(F^\letter)} \lambda_\point \point_{ij} 
		< \sum_{\letter=1}^m \sum_{\point\in V(F^\letter)} \lambda_\point \point_{i} 
		= \bundle_i = t_k. 
	\end{equation}
	Hence, we have $\hat{\point}_j = 1$ and so $\hat{\point}=\chi^{t_k}$. 
	Let $\hat{\point}=\mathbf{0}$ be the zero vector and $i\in [n]$ such
	that $\bundle_i = t_s$ attains the maximum value. Since
	$\hat{\point}_i = 0$, $\lambda_{\hat{\point}}>0$ and
	$\point_i\in\mset{0,1}^d$ for all $\point\in \bigcup_{\letter=1}^m
	V(F^\letter)$, we have 
	\begin{equation}\label[ineq]{ineq:case:0}
		t_s = \bundle_i =
		\sum_{\letter=1}^m \sum_{\point\in V(F^\letter)} \lambda_\point \point_{i} 
		< \sum_{\letter=1}^m \sum_{\point\in V(F^\letter)} \lambda_\point
		= \sum_{\letter=1}^m 1 = m
	\end{equation} 
	and therefore the number of characteristic vectors of the empty graph
	is $m- t_s > 0$, i.e. $\hat{\point}=\mathbf{0}\in \mset{\hat{\chi}^b \given
		b\in [m]} $.\\
	
	We now show $\bigcup_{\letter=1}^m V(F^\letter) \supseteq \left\{ \hat{\chi}^{b} \given b\in
	[m] \right\}$. Let $t_k$ be a non-zero value of $\bundle$. We show that
	there exists some $\point\in\bigcup_{\letter=1}^m V(F^\letter)$ such
	that $\point=\chi^{t_k}$. Let $i,j\in [n]$ such that $\bundle_i = t_k,
	\bundle_j = t_{k-1}$.
	Since $\bundle_i > \bundle_j$ and $\bigcup_{\letter=1}^m V(F^\letter)
	\subseteq \mset{\hat{\chi}^b \given b\in [m]}$, this implies
	\[\mset{ \chi^{t_l} \given t_l> t_{k-1}, i\in V(G_{t_l})} \cap
	\bigcup_{\letter=1}^m V(F^\letter) \neq \emptyset. \]
	Note that $i\in V(G_{t_l})$ implies that $\bundle_i \geq t_l$.  Thus,
	since otherwise $\bundle_i > t_k$, 
	\[\mset{ \chi^{t_l} \given t_l= t_{k}, i\in V(G_{t_l})} \cap
	\bigcup_{\letter=1}^m V(F^\letter) \neq \emptyset. \] 
	Suppose $\mathbf{0}\in \mset{\chi^b \given b\in [m]}$, i.e. there
	exists a characteristic vector of the empty graph and so $t_s <
	m$. Let $i\in [n]$ such that $\bundle_i = t_s$. \Cref{ineq:case:0}
	implies that there exists a $\hat{\point}\in  \bigcup_{b=1}^m V(F^b)$
	such that $\hat{\point}_i = 0$. By the above, we know that
	$\hat{\point}\in \mset{\hat{\chi}^b \given b\in [m]}$, which are
	characteristic vectors of nested graphs. In particular, all non-empty
	graphs contain the graph $G_{t_s}$ and hence the vertex $i$. Since
	$\hat{\point}$ is a characteristic vector of one of the nested graphs
	graphs, it must be the empty graph, i.e. $\hat{\point}=\mathbf{0}$. This shows that $\bigcup_{\letter=1}^m V(F^\letter) = \left\{ \hat{\chi}^{b} \given b\in
	[m] \right\}$. 
	
	Finally, we now show
	\[
	(t_l - t_{l-1}) \chi^{t_l} = \sum_{b=1}^{m} \sum_{\substack{q \in V(F^b) \\ q = \chi^{t_l}}} \lambda_q q
	\]  
	for all $l\in[s]$.
	Let $i\in [n]$ such that $a_i = t_l$ for some $l\in[s]$. Then $i\in V(G_{t_k})$ for all $k\leq l$ and $i\not \in V(G_{t_k})$ for all $k>l$. As $\bigcup_{\letter=1}^m V(F^\letter) = \left\{ \hat{\chi}^{b}
	\given j\in [m] \right\}$ and the graphs $G_{t_1},\dots,G_{t_s}$ are nested, we have that 
	\[
	t_l \chi^{t_k}_i = t_l = a_i = \sum_{b=1}^{m} \sum_{q \in V(F^b)} \lambda_q q_i = \sum_{b=1}^{m} \sum_{\substack{q \in V(F^b) \\ q \in \mset{\chi^{t_1},\dots,\chi^{t_l}}}} \lambda_q q_i.
	\]
	for all $k\leq l$. Hence,
	\[
	(t_l - t_{l-1}) \chi^{t_l}_i = (t_{l} \chi^{t_l}_i) - (t_{l-1} \chi^{t_{l-1}}_i) = \sum_{b=1}^{m} \sum_{\substack{q \in V(F^b) \\ q = \chi^{t_l}}} \lambda_q q_i.
	\]
	
	A similar statement holds for the edge $ij$ for $i,j\in V(G_{t_l})$. If $i\not\in V(G_{t_l})$ then 
	\[
	(t_l - t_{l-1}) \chi^{t_l}_i  = 0 = \sum_{b=1}^{m} \sum_{\substack{q \in V(F^b) \\ q = \chi^{t_l}}} \lambda_q q_i.
	\]
	and similarly for any $i,j\in[n]$ such that $ij$ is not an edge in $G_{t_l}$. Hence, all assumptions of 	\Cref{lem:perfect:matching} are fulfilled and implies that
	there is a labeling of the faces such that $\hat{\chi}^b \in F^b$ for
	all $b\in [m]$.
\end{proof} $ $

\subsection{Proofs of main results}

\begin{proof}[Proof of \Cref{thm:main}:]
	This is a reformulation of \Cref{th:const_size} for $\rr = 1$.
\end{proof}

\begin{proof}[Proof of \Cref{th:const_size}:]
	Let $\bundleStar\in \{0,\rr\}^n, m\geq \rr, a \in \piInverse(\bundleStar)$.
	By \Cref{prop:decomp:const}, for any faces $F^1, \dots, F^m$ of $P(K_n)$ such that $a \in \sum_{b=1}^{m} F^b$ holds that $a \in \sum_{b=1}^{m} \vertices(F^b)$.  \Cref{prop:key-long} implies that this holds if and only if for any set of valuations there exists an allocation and price $p\in\RR^d$ at which a competitive equilibrium exists.
\end{proof}

\begin{proof}[Proof of \Cref{th:CE}:]
	Let $\bundleStar\in \{0,1,\dots, m\}^n$.
	Then by \Cref{prop:CE_always_exists}, there exists a point $a \in \piInverse(\bundleStar)$, such that for any faces $F^1,\dots F^m$ of $P(K_n)$ containing $a \in \sum_{b=1}^{m} F^b$ holds that   $a \in \sum_{b=1}^{m} \vertices(F^b)$. Hence, by \Cref{cor:key-short},  for any set of valuations there exists an allocation and price at which a competitive equilibrium exists.
\end{proof}

\begin{proof}[Proof of \Cref{thm:main-general} and \Cref{thm:generalize-to-main}:]
	
	Let $\valuationtilde(\bundle) = \inner{\weighttilde,\bundle}$, where
	$\tilde{w}^b\in \RR^d$ is the vector in which we replace every entry of
	$\weight$ that has value $-\infty$ by the value $-M$ for a
	sufficiently large $M\in\RR$. 
	By assumption, the cliques $K_{\partition^b} ,b\in[m]$ form a
	covering of the graph $K_n$.
	Thus, we can find a partition of the vertices $[n]=V^1\sqcup\dots
	\sqcup V^m$ such that $V^b \subseteq \partition^b$. \\
	We consider the graph consisting of union of cliques, where the
	vertex sets of the cliques are the sets $V^1,\dots,V^m$.
	Let $\bundle\in \{0,1\}^d$ be the characteristic vector of this graph.
	According to \Cref{th:const_size}, given the valuations
	$\valuationtilde, b\in[m]$, we can find a price $\price\in\RR^d$ and
	an allocation $\bundle = \sum_{b=1}^m \bundle^b$ such that
	$\bundle^b \in \demand{\valuationtilde}{\price}$.
	More precisely, each $\bundle^b$ is the characteristic vector of the
	union of a subset of the cliques $V^1,\dots,V^b$. Furthermore, if
	$M$ is big enough then $\bundle^b \in
	\demand{\valuationtilde}{\price}$ implies that the graph
	corresponding to $\bundle^b$ is contained the graph $\partition^b$
	that agent $b$ has bid on (that is because if $a^b$ contains a
	vertex where $\weighttilde$ has value $-M$, then
	$\inner{\weighttilde-\price,\bundle^b} < 0 =
	\inner{\weighttilde-\price,\mathbf{0}} $). Thus, also the value
	$\inner{\weight-\price,\bundle^b}$ is finite and we have
	
	\[\inner{\weighttilde-\price,\bundle^b}=\inner{\weight-\price,\bundle^b}.\]
	Since in general we always have 
	\[\inner{\weighttilde-\price,\bundle}\geq \inner{\weight-\price,\bundle}\]
	for any $\bundle\in P(K_n)$, it follows that $\bundle^b \in
	\demand{\valuation}{\price}$ for the original valuations, for all
	$j\in [m]$, so competitive equilibrium at $a$ always exists.
\end{proof}

\newpage
\bibliographystyle{alpha}
\bibliography{references}

\begin{thebibliography}{COP15b}

\bibitem[BFS20]{bichler2021walrasian}
Martin Bichler, Maximilian Fichtl, and Gregor Schwarz.
\newblock Walrasian equilibria from an optimization perspective: A guide to the
  literature.
\newblock {\em Naval Research Logistics (NRL)}, 68(4):496--513, November 2020.

\bibitem[BK19]{Baldwin_Klemperer_Understanding_Preferences}
Elizabeth Baldwin and Paul Klemperer.
\newblock Understanding preferences: {\textquotedblleft}demand
  types{\textquotedblright}, and the existence of equilibrium with
  indivisibilities.
\newblock {\em Econometrica}, 87(3):867--932, 2019.

\bibitem[BKS15]{borgers2015introduction}
Tilman Börgers, Daniel Krähmer, and Roland Strausz.
\newblock {\em An Introduction to the Theory of Mechanism Design}.
\newblock Oxford University Press, New York, July 2015.

\bibitem[COP15a]{candogan_parrilo_trees}
Ozan Candogan, Asuman Ozdaglar, and Pablo~A. Parrilo.
\newblock Iterative auction design for tree valuations.
\newblock {\em Operations Research}, 63(4):751--771, June 2015.

\bibitem[COP15b]{candogan_parrilo_pricing_equilibria}
Ozan Candogan, Asuman Ozdaglar, and Pablo~A. Parrilo.
\newblock Pricing equilibria and graphical valuations.
\newblock {\em SSRN Electronic Journal}, 01(01):1--26, January 2015.

\bibitem[DK04]{danilov2004discrete}
Vladimir~I. Danilov and Gleb~A. Koshevoy.
\newblock Discrete convexity and unimodularity. {I}.
\newblock {\em Advances in Mathematics}, 189(2):301--324, December 2004.

\bibitem[Jos]{joswig}
Michael Joswig.
\newblock Essentials of tropical combinatorics.
\newblock In preparation.

\bibitem[Lem17]{leme2017gross}
Renato~Paes Leme.
\newblock Gross substitutability: An algorithmic survey.
\newblock {\em Games and Economic Behavior}, 106:294--316, November 2017.

\bibitem[LRS10]{loera_triangulations}
Jesus~De Loera, Joerg Rambau, and Francisco Santos.
\newblock {\em Triangulations}.
\newblock Springer-Verlag GmbH, Heidelberg, August 2010.

\bibitem[Nis15]{nisan2015algorithmic}
Noam Nisan.
\newblock Algorithmic mechanism design: Through the lens of multiunit auctions.
\newblock In {\em Handbook of Game Theory with Economic Applications},
  volume~4, pages 477--515. Elsevier, Boston, 2015.

\bibitem[Pad89]{Padberg89}
Manfred Padberg.
\newblock The boolean quadric polytope: Some characteristics, facets and
  relatives.
\newblock {\em Mathematical Programming}, 45(1-3):139--172, August 1989.

\bibitem[RTC15]{roughgarden2015prices}
Tim Roughgarden and Inbal Talgam-Cohen.
\newblock Why prices need algorithms.
\newblock In {\em Proceedings of the Sixteenth {ACM} Conference on Economics
  and Computation}, pages 19--36, New York, NY, USA, June 2015. Association for
  Computing Machinery.

\bibitem[Sch98]{schrijver_theorylinearinteger}
Alexander Schrijver.
\newblock {\em Theory of Linear Integer Programming}.
\newblock John Wiley \& Sons, Chichester; Weinheim, June 1998.

\bibitem[TY19]{Tran_Yu_Product-Mix_Auctions}
Ngoc~Mai Tran and Josephine Yu.
\newblock Product-mix auctions and tropical geometry.
\newblock {\em Mathematics of Operations Research}, 44(4):1396--1411, November
  2019.

\bibitem[Zie00]{lectures}
G{\"u}nter~M. Ziegler.
\newblock Lectures on 0/1-polytopes.
\newblock In Gil Kalai and G{\"u}nter~M. Ziegler, editors, {\em Polytopes ---
  Combinatorics and Computation}, pages 1--41. Birkh{\"a}user Basel, Basel,
  2000.

\end{thebibliography}

\end{document}